\documentclass[letterpaper,11pt]{article}
\usepackage{amsmath,amssymb,amsthm}
\usepackage{fullpage}

\begin{document}

\makeatletter
\newtheorem*{rep@theorem}{\rep@title}
\newcommand{\newreptheorem}[2]{%
\newenvironment{rep#1}[1]{%
 \def\rep@title{#2 \ref{##1}}%
 \begin{rep@theorem}}%
 {\end{rep@theorem}}}
\makeatother

\newtheorem{thm}{Theorem}
\newreptheorem{thm}{Theorem}
\newtheorem{prop}{Proposition}
\newtheorem{cor}{Corollary}
\newreptheorem{cor}{Corollary}
\newtheorem{lem}{Lemma}
\newreptheorem{lem}{Lemma}
\newtheorem{quest}{Question}
\newtheorem*{conj*}{Conjecture}
\newtheorem*{thm*}{Theorem}

\theoremstyle{definition}
\newtheorem{defn}{Definition}

\theoremstyle{remark}
\newtheorem{rem}{Remark}

\newcommand{\Rho}{\mathrm{P}}
\newcommand{\EE}{\mathbb{E}}
\newcommand{\PP}{\mathbb{P}}

\title{General divisor function inequalities and the third cumulant}
\author{Zarathustra Brady}
\date{}
\maketitle

\begin{abstract} We extend a lower bound of Munshi on sums over divisors of a number $n$ which are less than a fixed power of $n$ from the squarefree case to the general case. In the process we prove a lower bound on the entropy of a geometric distribution with finite support, as well as a lower bound on the probability that a random variable is less than its mean given that it satisfies a natural condition related to its third cumulant.
\end{abstract}

\section{Introduction}

We consider the following problem: for which $\beta, \delta, s$ does the inequality
\begin{align}\label{main}
\tau(n)^s \ll_{\beta,\delta,s} \sum_{\substack{d\mid n\\ d\le n^{\delta}}} \tau(d)^{\beta}
\end{align}
hold for all positive integers $n$? Munshi \cite{munshi} solves this problem for $s=1$ and $n$ squarefree: if $0 < \delta \le \frac{1}{2}$ and
\[
\beta > \frac{1-H(\delta)}{\delta},
\]
where
\[
H(\delta) = \delta\log_2\left(\frac{1}{\delta}\right) + (1-\delta)\log_2\left(\frac{1}{1-\delta}\right),
\]
then inequality \eqref{main} holds for squarefree $n$. In fact, Munshi's argument easily generalizes to any $s$ if we require that
\[
\beta > \frac{s-H(\delta)}{\delta},
\]
and this is best possible by the same reasoning as in \cite{munshi}.

In this paper we generalize Munshi's argument to arbitrary natural numbers $n$. Our main result, proved in section 4, is the following.

\begin{repthm}{result} If $0 < \delta \le \frac{1}{2}$, $\beta,s \ge 0$ satisfy
\[
\beta > \frac{s - H(\delta)}{\delta},
\]
then
\[
\tau(n)^s \ll_{\beta,\delta,s} \sum_{\substack{d\mid n\\ d\le n^{\delta}}} \tau(d)^{\beta}.
\]
\end{repthm}

The main new idea in the proof is to sample divisors $d$ of $n$ from a probability distribution having high entropy, while keeping the average value of $\log(d)$ small. A crucial ingredient in the proof is the following entropy inequality, which is proved in section 3.

\begin{repcor}{geo} If $X$ is geometrically distributed on $\{0, ..., m\}$ with mean $\delta m$ then
\[
H(X) \ge \log_2(m+1)H(\delta),
\]
and the inequality is strict if $m > 1$ and $\delta \not\in \{0, \frac{1}{2}, 1\}$.
\end{repcor}

In the process of proving our main result, we also prove a variation on a related inequality due to Soundararajan. Suppose that $\delta$ is a real number between $0$ and $1$, and define $c(\delta)$ to be the largest real number such that, for any squarefree number $n$, we have the inequality
\begin{align}\label{secondary}
\sum_{\substack{d\mid n\\ d\le n^{\delta}}} \delta^{\omega(d)}(1-\delta)^{\omega(n/d)} \ge c(\delta).
\end{align}
Taking $n$ to have $k$ prime factors that are sufficiently close in size, we see that
\[
\delta < \frac{1}{k} \implies c(\delta) \le (1-\delta)^k.
\]
Soundararajan has shown in \cite{sound-mult} (with different notation - his $A(t)$ is our $c(1/(1+t))$, and his $B(t)$ is our $1-c(t/(1+t))$) the following recursive inequalities:
\begin{align*}
c\left(\frac{\delta}{1+\delta}\right) &\ge \frac{c(\delta)}{1+c(\delta)},\\
c\left(\frac{1}{2-\delta}\right) &\ge \frac{c(\delta)}{1+c(\delta)}.
\end{align*}
Using these together with the obvious bound $c(1) = 1$, he shows that $c(1-1/k) = 1/k$ for $k \in \mathbb{N}$, and that if $\delta$ is rational with continued fraction $[a_0,a_1,...,a_r]$ then
\[
c(\delta) \ge \frac{1}{a_0 + \cdots + a_r}.
\]

\begin{defn} For any integers $n \ge k$, define $g(n,k)$ by
\[
g(n,k) = \min_{a_1 + \cdots + a_n = 0} \big|\big\{S\subseteq \{1,...,n\}\big| |S| = k, \sum_{i\in S}a_i \ge 0\big\}\big|.
\]
\end{defn}

\begin{conj*}[Manickam-Mikl\'{o}s-Singhi] For $n \ge 4k$, $g(n,k) = \binom{n-1}{k-1}$.
\end{conj*}
The Manickam-Mikl\'{o}s-Singhi conjecture has been proved in the cases $k \le 7$ \cite{mms}, $n\ge 2k^3$, $n \ge 33k^2$ \cite{mms-quad}, and $n \ge 10^{46}k$ \cite{mms-linear}. In \cite{mms}, a slightly stronger conjecture is made based on numerical evidence: if $\binom{n-1}{k-1} \le \binom{n-3}{k}$, then $g(n,k) = \binom{n-1}{k-1}$. The MMS conjecture is related to $c(\delta)$ by the following easy proposition.

\begin{prop} If $\delta = \frac{a}{b}$, then
\[
c(\delta) \ge \frac{g(b,a)}{\binom{b}{a}}.
\]
In particular, if $g(b,a) = \binom{b-1}{a-1}$ then $c(\delta) \ge \delta$. If the stronger version of the MMS conjecture proposed in \cite{mms} holds, then $c(\delta) \ge \delta$ for all $\delta \le (1-\delta)^3$.
\end{prop}

In the next section, we prove that
\[
\delta \le \frac{1}{2} \implies c(\delta) \ge \frac{1}{2e^{3/2}}.
\]
In fact, we prove the following stronger claim.

\begin{repcor}{quantile} Let $X_1, ..., X_n$ be independent random variables supported on $\mathbb{N}$ such that for each $i$ the function $k\mapsto\PP[X_i = k]$ is decreasing, and let $w_1, ..., w_n \ge 0$. Let $X = \sum_{i=1}^n w_iX_i$. Then $\PP[X \le \EE[X]] \ge \frac{1}{2e^{3/2}}$.
\end{repcor}

\section{Lower bound on the probability that a random variable is less than its mean}

The arguments in this section are inspired by a MathOverflow post of fedja \cite{fedja}, which used Bernstein's trick in a similar way to solve a closely related problem.

Consider the following property which a random variable $X$ might have:
\begin{align}
\forall t \ge 0\;\;\; &\frac{d^3}{dt^3}\log\left(\EE[e^{-tX}]\right) \le 0.\label{strong}\tag{P}
\end{align}
If independent random variables $X_1, ..., X_n$ all have property \eqref{strong}, and if $w_1, ..., w_n \ge 0$, then the random variable $X = \sum_i w_iX_i$ also has property \eqref{strong} by Bernstein's trick.

Note that when $t = 0$, property \eqref{strong} says that the third cumulant of $X$, $\EE[(X-\EE[X])^3]$, is at least zero.

\begin{thm}\label{low} If a random variable $X$ has property \eqref{strong}, then
\[
\PP[X \le \EE[X]] \ge \frac{1}{2e^{3/2}}.
\]
\end{thm}
\begin{proof} Let $Y = \EE[X] - X$, and define the function $g : \mathbb{R}_{\ge 0} \rightarrow \mathbb{R}$ by
\[
g(t) = \log\left(\EE[e^{tY}]\right).
\]
Property \eqref{strong} says that $g''(t)$ is decreasing for $t\ge 0$. Since $\EE[Y] = 0$, we also have $g'(0) = 0$, and so for any $t \ge 0$ we have
\[
tg''(t) \le \int_0^t g''(x)dx = g'(t).
\]
Integrating this we see that $tg'(t) \le 2g(t)$. The inequality $tg''(t) \le g'(t)$ is easily seen to be equivalent to
\[
\frac{\EE[(tY)^2e^{tY}]}{\EE[tYe^{tY}]} \le \frac{\EE[tYe^{tY}]}{\EE[e^{tY}]} + 1.
\]
If $g$ is not identically $0$ then we can find $t$ such that $tg'(t) = 1$, or equivalently such that $\EE[tYe^{tY}] = \EE[e^{tY}] = e^{g(t)}$. For this $t$ we have
\[
\EE\left[tY\left(8-3tY\right)e^{tY}\right] \ge \left(8-6\right)e^{g(t)} \ge 2e^{1/2}.
\]
The function $p(x) = x(8-3x)e^x$ has $p(x) \le 0$ for $x \le 0$ and $p(x) \le p(2) = 4e^2$ for all $x$, so by Markov's inequality
\[
\PP[X \le \EE[X]] = \PP[Y \ge 0] \ge \frac{1}{2e^{3/2}}.\qedhere
\]
\end{proof}

\begin{thm}\label{decrease} Suppose that the random variable $X$ is supported on $\mathbb{N}$, with $\PP[X = k]$ a decreasing function of $k$. Then $X$ satisfies property \eqref{strong}.
\end{thm}
\begin{proof} Expanding property \eqref{strong}, it becomes
\[
\EE[X^3e^{-tX}]\EE[e^{-tX}]^2 + 2\EE[Xe^{-tX}]^3 \ge 3\EE[X^2e^{-tX}]\EE[Xe^{-tX}]\EE[e^{-tX}].
\]
Setting $a_k = \PP[X = k]$ and $x = e^{-t}$ we get the polynomial inequality
\[
\sum_{i,j,k} a_ia_ja_kx^{i+j+k}(i^3 + 2ijk - 3i^2j) \ge 0,
\]
which we need to check for $a_0 \ge a_1 \ge \cdots \ge 0$ and $1 \ge x \ge 0$. The left hand side of the above is equal to
\[
\sum_{i < j} a_ia_jx^{i+j}(a_ix^i - a_jx^j)(j-i)^3 + \sum_{i<j<\frac{i+k}{2}} a_ia_kx^{i+k}(a_jx^j - a_{i+k-j}x^{i+k-j})(i+k-2j)(j+k-2i)(2k-i-j),
\]
which is obviously nonnegative.
\end{proof}

\begin{cor}\label{quantile} Let $X_1, ..., X_n$ be independent random variables supported on $\mathbb{N}$ such that for each $i$ the function $k\mapsto\PP[X_i = k]$ is decreasing, and let $w_1, ..., w_n \ge 0$. Let $X = \sum_{i=1}^n w_iX_i$. Then $\PP[X \le \EE[X]] \ge \frac{1}{2e^{3/2}}$.
\end{cor}

\begin{cor}\label{soundy} Let $\delta \le \frac{1}{2}$, and let $f:\mathbb{N}\rightarrow [0,\infty)$ be a nonnegative multiplicative function such that for every prime $p$ we have $\frac{f(p)}{f(p)+1} \le \delta$. Then for any squarefree number $n$ we have
\[
\sum_{\substack{d\mid n\\ d\le n^{\delta}}} f(d) \ge \frac{1}{2e^{3/2}}\sum_{d\mid n} f(d).
\]
\end{cor}

\section{A lower bound for the entropies of certain probability distributions}

Let $X$ be a random variable supported the set $\{0, ..., m\}$, with probability distribution $\rho = (\rho_0, ..., \rho_m)$. We define the entropy of $X$ to be
\[
H(X) = \sum_{i=0}^m \rho_i\log_2\left(\frac{1}{\rho_i}\right).
\]

In the next section, we will need the existence of a random variable $X$ as above with $\EE[X] = \delta m$ given and $H(X)$ large. It's a well-known fact that the optimal choice of $X$ will be geometrically distributed. Unfortunately the entropy of a geometric distribution on a finite set, as a function of the mean, is quite complicated and directly proving a lower bound is rather difficult. Instead, we will inductively construct probability distributions which are simpler to analyze and still have sufficiently large entropy.

\begin{lem}\label{ent} For every $m \ge 1$ and every $0 \le \delta \le 1$ there is a random variable $X$ supported on the set $\{0, ..., m\}$ which has mean $\delta m$ and entropy satisfying the inequality
\[
H(X) \ge \log_2(m+1)H(\delta).
\]
\end{lem}
\begin{proof} It's enough to prove this for $0 < \delta \le \frac{1}{2}$. We proceed by induction on $m$. The case $m = 1$ is immediate. For $m > 1$, we let $Y$ be a random variable on the set $\{0, ..., m-1\}$ with mean $\delta(m-1)$, satisfying $H(Y) \ge \log_2(m)H(\delta)$. Define $X$ to be $0$ with probability $\frac{1-\delta}{m\delta+1-\delta}$, and to be $1+Y$ with probability $\frac{m\delta}{m\delta+1-\delta}$. Then
\[
\mathbb{E}[X] = \frac{m\delta}{m\delta+1-\delta}\left(1+\mathbb{E}[Y]\right) = \frac{m\delta}{m\delta+1-\delta}(1+\delta(m-1)) = \delta m,
\]
and
\[
H(X) = H\left(\frac{1-\delta}{m\delta+1-\delta}\right) + \frac{m\delta}{m\delta+1-\delta}H(Y) \ge H\left(\frac{1-\delta}{m\delta+1-\delta}\right) + \frac{m\delta}{m\delta+1-\delta}\log_2(m)H(\delta).
\]
It suffices to show that the right hand side of the above is at least $\log_2(m+1)H(\delta)$.

Making the change of variables $x = \frac{1-\delta}{m\delta}$, we just need to show
\[
H\left(\frac{1}{x+1}\right) \ge \left(\log_2(m+1) - \frac{\log_2(m)}{x+1}\right)H\left(\frac{1}{mx+1}\right)
\]
for real numbers $m, x$ satisfying $m \ge 1$ and $mx \ge 1$. Since we clearly have equality when $m = 1$ or $m = \frac{1}{x}$, it is enough to show that the right hand side is a decreasing function of $m$. Taking the derivative with respect to $m$, using the identity $H'(\delta) = \log_2\left(\frac{1-\delta}{\delta}\right)$, we see that we just need to check
\[
\left(\log_2(m+1) - \frac{\log_2(m)}{x+1}\right)\frac{x\log_2(mx)}{(mx+1)^2} \ge \log_2(e)\left(\frac{1}{m+1}-\frac{1}{m+mx}\right)H\left(\frac{1}{mx+1}\right).
\]
Changing variables back to $m, \delta$ and rearranging, this becomes
\[
(1-\delta)\left(1+\frac{1}{m}\right)\log_2(m+1) + \delta(m+1)\log_2\left(1+\frac{1}{m}\right) \ge \frac{\log_2(e)(1-2\delta)H(\delta)}{\delta(1-\delta)\log_2\left(\frac{1-\delta}{\delta}\right)}.
\]
From $(1-\delta) \ge \delta$ and $m \ge 1$ we easily deduce that the left hand side is at least $2$, so it is enough to prove the single variable inequality
\[
2\delta(1-\delta)\frac{\log(1-\delta)-\log(\delta)}{(1-\delta)-\delta} \ge H(\delta),
\]
where the logarithms on the left hand side are taken to the base $e$. We leave this inequality as an exercise for the reader.
\end{proof}

\begin{cor}\label{geo} If $X$ is geometrically distributed on $\{0, ..., m\}$ with mean $\delta m$ then
\[
H(X) \ge \log_2(m+1)H(\delta),
\]
and the inequality is strict if $m > 1$ and $\delta \not\in \{0, \frac{1}{2}, 1\}$.
\end{cor}
\begin{proof} This follows from the previous lemma together with the well-known fact that a geometric distribution has the maximum entropy among all distributions on a finite set which have a given mean.
\end{proof}

\begin{rem} In the case $m+1 = 2^k$ we can give a much simpler proof of Lemma \ref{ent}. Let $B_0, ..., B_{k-1}$ be i.i.d. random variables which are each $0$ with probability $1-\delta$ and $1$ with probability $\delta$. Then if we take
\[
X = \sum_{i=0}^{k-1} 2^iB_i,
\]
we have $\mathbb{E}[X] = \delta m$ and $H(X) = kH(\delta)$. This probability distribution corresponds to a trick used by Wolke in \cite{wolke}.
\end{rem}

\section{Divisor sum inequalities}

\begin{thm}\label{result} If $0 < \delta \le \frac{1}{2}$, $\beta,s \ge 0$ satisfy
\[
\beta > \frac{s - H(\delta)}{\delta},
\]
then
\[
\tau(n)^s \ll_{\beta,\delta,s} \sum_{\substack{d\mid n\\ d\le n^{\delta}}} \tau(d)^{\beta}.
\]
\end{thm}
\begin{proof} Choose a number $M$ such that for all $m \ge M$ we have
\[
\beta > \frac{s - \frac{\lfloor\log_2(m+1)\rfloor}{\log_2(m+1)}H(\delta)}{\delta}.
\]

Write $n = \prod_i p_i^{m_i}$. We define a collection of independent random variables $X_i$, $X_i$ taking values in $\{0, ..., m_i\}$, as follows. If $m_i < M$, we take $X_i$ to be geometrically distributed with mean $\delta m_i$. If $m_i \ge M$, choose $k$ such that $2^k-1 \le m_i < 2^{k+1}-1$, and let $B_0, ..., B_{k-1}$ be $k$ i.i.d. random variables which are each $0$ with probability $1-\delta$ and $1$ with probability $\delta$. Set
\[
X_i = \Bigg(\sum_{j=0}^{k-2} 2^jB_j\Bigg) + (m_i+1-2^{k-1})B_{k-1}.
\]
Finally, we define a random variable $D$ dividing $n$ by $D = \prod_i p_i^{X_i}$.

We have
\[
\mathbb{E}[\log(D)] = \sum_i \mathbb{E}[X_i]\log(p_i) = \delta\log(n),
\]
so by Corollary \ref{quantile} we have
\[
\mathbb{P}[D \le n^{\delta}] \ge \frac{1}{2e^{3/2}}.
\]
Setting $P_n(d) = \mathbb{P}[D=d]$, this can be written as
\[
1 \le 2e^{3/2}\sum_{\substack{d\mid n\\ d\le n^{\delta}}} P_n(d).
\]

By H\"{o}lder's inequality, for any $t > 0$ we have
\[
\sum_{\substack{d\mid n\\ d\le n^{\delta}}} P_n(d) \le \bigg(\sum_{d\mid n,\ d\le n^{\delta}} \tau(d)^{\beta}\bigg)^{\frac{1}{t+1}}\bigg(\sum_{d\mid n} P_n(d)^{\frac{t+1}{t}}\tau(d)^{-\frac{\beta}{t}}\bigg)^{\frac{t}{t+1}}.
\]
Combining the last two inequalities, we see that
\[
\bigg(\sum_{d\mid n} P_n(d)^{\frac{t+1}{t}}\tau(d)^{-\frac{\beta}{t}}\bigg)^{-t} \le \big(2e^{3/2}\big)^{t+1}\sum_{\substack{d\mid n\\ d\le n^{\delta}}} \tau(d)^{\beta}.
\]
To finish, we just need to choose $t$ large enough that the left hand side of the above is at least $\tau(n)^s$. Since the left hand side is a multiplicative function of $n$, we can restrict to the case $n = p^m$, with just a single probability distribution $X$ on the possible exponents $\{0, ..., m\}$. Write $\rho_m(x)$ for $P_{p^m}(p^x) = \mathbb{P}[X=x]$. Then we just need to choose $t$ large enough to make the inequality
\begin{align}\label{tin}
(m+1)^s \le \left(\sum_{x=0}^m \rho_m(x)^{\frac{t+1}{t}}(x+1)^{-\frac{\beta}{t}}\right)^{-t}
\end{align}
hold for all $m\ge 1$. We have
\[
\lim_{t\rightarrow\infty}\left(\sum_{x=0}^m \rho(x)\left(\frac{(x+1)^{\beta}}{\rho(x)}\right)^{-\frac{1}{t}}\right)^{-t} = \prod_{x=0}^m \left(\frac{(x+1)^{\beta}}{\rho(x)}\right)^{\rho(x)} = 2^{H(X)+\beta \mathbb{E}[\log_2(X+1)]}.
\]
Since $\log_2$ is a concave function, we have
\[
\mathbb{E}[\log_2(X+1)] \ge \mathbb{E}\left[\frac{X}{m}\log_2(m+1) + \left(1-\frac{X}{m}\right)\log_2(1)\right] = \delta\log_2(m+1).
\]
Thus, by the assumption on $\beta$ and Corollary \ref{geo} we can find a $t_0$ such that for any $t \ge t_0$ and any $m < M$ inequality (\ref{tin}) is satisfied. For $m \ge M$, we use the easy inequality
\[
\left(\sum_{x=0}^m \rho(x)^{\frac{t+1}{t}}(x+1)^{-\frac{\beta}{t}}\right)^{-t} \ge \left((1-\delta)^\frac{t+1}{t}+\delta^{\frac{t+1}{t}}2^{-\frac{\beta}{t}}\right)^{-t\lfloor \log_2(m+1)\rfloor},
\]
which follows from the fact that for any $x$, $x+1$ is at least $2^B$, where $B$ is the number of $1$s in the binary representation of $x$. Thus if we take $t$ large enough to make
\[
\left((1-\delta)^\frac{t+1}{t}+\delta^{\frac{t+1}{t}}2^{-\frac{\beta}{t}}\right)^{-t}
\]
sufficiently close to $2^{H(\delta)+\beta\delta}$, then inequality (\ref{tin}) will be satisfied for $m \ge M$ as well.
\end{proof}

\bigskip

\emph{Acknowledgements.} The author would like to thank Professor Soundararajan for introducing the author to this problem and providing helpful feedback on early drafts of this paper, as well as pointing out that the idea in fedja's MathOverflow post could be used to simplify the argument and improve the bounds.

\nocite{vandercorput}
\bibliography{octahedral}
\bibliographystyle{plain}

\end{document}